\documentclass[a4paper,12pt]{amsart}

\usepackage{amsmath,amssymb}
\usepackage{graphicx}
\usepackage{subfig}

\usepackage{hyperref}
\usepackage{fullpage}
\newtheorem{theo}{Theorem}[section]
\newtheorem{lemma}[theo]{Lemma}
\newtheorem{coro}[theo]{Corollary}
\newtheorem{prop}[theo]{Proposition}
\newtheorem{assump}[theo]{Assumption}
\theoremstyle{definition}
\newtheorem{defi}[theo]{Definition}

\theoremstyle{remark}
\newtheorem{rem}[theo]{Remark}

\def\<#1,#2>{\langle #1,#2\rangle}

\renewcommand{\geq}{\geqslant}
\renewcommand{\leq}{\leqslant}

\newcommand{\W}{\mathcal{W}}
\newcommand{\M}{\mathcal{M}}
\newcommand{\R}{\mathbb{R}}

\newcommand{\pP}{\mathcal{P}}
\newcommand{\K}{\mathcal{K}}

\newcommand{\D}{\mathcal{D}}

\newcommand{\sym}{\operatorname{S}}

\newcommand{\conv}{\operatorname{conv}}
\newcommand{\firstdef}[1]{{\em #1}}

\title{Contraction of Riccati flows applied to the convergence analysis of a max-plus curse of dimensionality free method}

\author{Zheng QU}
\address{CMAP and INRIA\\
         \'Ecole Polytechnique\\
          91128 Palaiseau C\'edex, France}
\email[]{zheng.qu@polytechnique.edu}
\thanks{A short manuscript announcing the present results (without the proofs) has been submitted to the European control conference 2013.}
\keywords{Dynamic programming, maxplus basis numerical method, contraction mapping in Thompson's metric, indefinite Riccati flow,
switching linear quadratic control}
\subjclass[2010]{Primary 49M25,49L20,47H09; Secondary 90C59, 93C30,49N10.}

\begin{document}
\maketitle

\begin{abstract}
Max-plus based methods have been recently explored for solution of first-order Hamilton-Jacobi-Bellman equations by several authors. In particular, McEneaney's curse-of-dimensionality free method applies to the equations where the Hamiltonian takes the form of a (pointwise) maximum of linear/quadratic forms.
 In previous works of McEneaney and Kluberg, the approximation error of the method was shown to be $O(1/(N\tau))$+$O(\sqrt{\tau})$ where $\tau$ is the time discretization step and $N$ is the number of iterations. Here we use a recently established contraction result for the indefinite Riccati flow in Thompson's metric to show that under different technical assumptions, still covering an important class of problems, the error is only of order  $O(e^{-\alpha N\tau})+O(\tau)$ for some $\alpha>0$. This also allows us to obtain improved estimates of the execution time and to tune the precision of the pruning procedure, which in practice is a 
critical element of the method.
\end{abstract}

\section {Introduction}

\subsection{Max-plus methods in optimal control}
Dynamic Programming (DP) is a general approach to the solution of optimal control problems. 
In the case of deterministic optimal control, this approach leads to solving a first-order, nonlinear partial differential equation, the Hamilton-Jacobi-Bellman equation(HJB PDE). 
Various methods have been proposed for solving the HJB PDE. We cite among all the finite difference schemes, 
the method of the vanishing viscosity of Crandall and Lions~\cite{crandall-lions}, the discrete dynamic programming method or
semi-Lagrangian method developed by Falcone~\cite{falcone} and others~\cite{capuzzodolcetta,
falcone-ferretti,
carlini-falcone-ferretti}, the high order ENO  schemes introduced by Osher, Sethian and Shu~\cite{OsherSethian88,OsherShu91},
 the discontinuous Galerkin method by Hu and Shu~\cite{HuShu99}, the ordered upwind methods for convex static Hamilton-Jacobi equations by Sethian and Vladimirsky~\cite{SethianVladimirsky03} which is an extension of the fast marching method for the Eikonal equations~\cite{Sethian99}, and
the antidiffusive schemes for advection of Bokanowski and Zidani~\cite{zidani-bokanowski}. 
  These methods generally require the generation of a grid on the state space. This is known to suffer from the so called curse-of-dimensionality since the computational growth in the state-space dimension is exponential.

Recently a new class of methods has been developed after
the work of Fleming and McEneaney~\cite{a5}, see in particular the works of McEneaney~\cite{curseofdim}, of Akian, Gaubert and Lakhoua~\cite{a6}, of McEneaney, Deshpande and Gaubert~\cite{a7}, of Dower and McEneaney~\cite{DowerMcE} and of James \textit{et al.}~\cite{eneaneyphys}. These methods are referred to as \textit{max-plus basis methods} since they all rely on max-plus algebra.
Their common idea is to approximate the value function by a supremum
of finitely many ``basis functions'' and to propagate the supremum forward in time by exploiting the max-plus
linearity of the Lax-Oleinik semigroup. Recall that the Lax-Oleinik semigroup $(S_t)_{t\geq0}$ associated to a Hamiltonian $H(\cdot,\cdot):\R^n\times \R^n\rightarrow \R$ is 
the evolution semigroup of the following HJB PDE
\begin{equation}\label{guih}
 -\frac{\partial v}{\partial t}+H(x,\nabla v)=0,\qquad (x,t)\in \R^n\times (0,T],
\end{equation}
with initial condition
\begin{equation}
 v(x,0)=\phi(x),\qquad  x\in \R^n .
\end{equation}
Thus, $S_t$ maps the initial function $\phi(\cdot)$ to the function $v(\cdot,t)$. Among several max-plus basis methods which have been proposed, the curse-of-dimensionality-free method introduced by McEneaney~\cite{curseofdim} is of special interest. This method applies to the special class of HJB PDE where the Hamiltonian $H$ is given or approximated as a pointwise maximum of computationally simpler Hamiltonians:
\begin{align}\label{a-H-eq}
H(x,\nabla V)=\max_{m\in \M} \{H^ m(x,\nabla V)\}
\end{align}
with $\M=\{1,2,\cdots,M\}$. In particular, McEneaney studied the case where each Hamiltonian $H^ m$ is a linear/quadratic form, originating from a linear quadratic optimal control problem:
$$
H^ m(x,p)=(A^ mx)'p+\frac{1}{2}x'D^ mx+\frac{1}{2}p'\Sigma^ m p,
$$
where $(A^m, D^m,\Sigma^m)$ are matrices meeting certain conditions. Such Hamiltonian $H$ corresponds to a linear quadratic switching optimal control problem (Section~\ref{sec-probclass}) where the control switches between several linear quadratic systems. We are interested in finding the value function $V$ of the corresponding infinite horizon switching optimal control problem. The method consists in two successive approximations (Section~\ref{sec-method}).
First we approximate the infinite horizon problem by a finite horizon problem. 
Then we approximate the value function of the finite horizon switching optimal control problem by choosing an optimal strategy which does not switch on small intervals. 

We denote by $(S_t)_{t\geq 0}$ and $(S_t^m)_{t\geq 0}$ for all $m\in \M$ respectively the semigroup corresponding to $H$ and $H^m$ for all $m\in \M$. Let $V^0$ be a given initial function and $T>0$ be the finite horizon.
The first approximation uses $S_T[V^0]$ to approximate $V$ and introduces the finite-horizon truncation error:
  $$\epsilon_0(x,T,V^0):=V(x)-S_T[V^0](x).$$ 
Let $\tau>0$ be a small time step and $N>0$ such that $T=N\tau$. Denote by $\tilde S_{\tau}$ the semigroup of the optimal control problem where the control does not switch on the interval $[0,\tau]$.
The second approximation approximates $S_T[V^0]$ by $\{\tilde S_\tau\}^N[V^0]$. The error at point $x$  of this time discretization approximation is denoted by:
$$
\epsilon(x,\tau, N, V^0):=S_T[V^0](x)-\{\tilde S_\tau\}^N[V^0](x).
$$
 The total error at a point $x$ is then simply $\epsilon_0(x,T,V^0)+\epsilon(x,\tau, N, V^0)$. We shall see that $\tilde S_\tau=\sup_{m\in \M} S^m_\tau$. Therefore $\tilde S_\tau$ applied to a quadratic function corresponds to solving $|\M|$ Riccati equations, requiring $O(|\M|n^3)$ arithmetic operations.
The total number of computational cost is $O(|\M|^Nn^ 3)$, with a cubic growth in the state dimension $n$. In this sense it is considered as a curse of dimensionality free method. However, we see that the computational cost is bounded by a number exponential in the number of iterations, which is 
referred to as the curse of complexity. 
In practice, a pruning procedure denoted by $\pP_\tau$ removing at each iteration a number of functions less useful than others is needed in order to reduce the curse of complexity. We denote the error at point $x$ of the time dicretization approximation incorporating the pruning procedure by:
$$
\epsilon^\mathcal{P_{\tau}}(x,\tau, N,V^0)=S_T[V^0](x)-\{\pP_\tau\circ\tilde S_\tau\}^N[V^0](x).
$$

\subsection{Main contributions}
In this paper, we analyze the growth rate of $\epsilon_0(x,T,V^0)$ as $T$ tends to infinity and the growth rate of $\epsilon^{\pP_\tau}(x,\tau, N, V^0)$ as $\tau$ tends to $0$, incorporating a pruning procedure $\pP_\tau$ of error $O(\tau^r)$ with $r>1$. The error $\epsilon(x,\tau,N,V^0)$ in the absence of pruning is obtained when $r=+\infty$.

We show that under technical assumptions (Assumption~\ref{assum1} and~\ref{assum2}), 
$$
\epsilon_0(x,T,V^0)=O(e^{-\alpha T}),\quad \mathrm{as~~}T\rightarrow +\infty
$$
 uniformly for all $x\in \R^n$ and all initial quadratic functions $V^0(x)=\frac{1}{2}x'Px$ where $P$ is a matrix in a certain compact (Theorem~\ref{theo-error1}). We also show that given a pruning procedure generating an error $O(\tau^ r)$ with $r>1$,
$$
\epsilon^\mathcal{P_{\tau}}(x,\tau, N,V^0)=O(\tau^{\min\{1,r-1\}}), \quad \mathrm{as~~}\tau\rightarrow 0
$$
uniformly for all $x\in \R^n$, $N\in \mathbb N$ and $V^0$ as above
(Theorem~\ref{theo-error2}). 
As a direct corollary, we have
$$
\epsilon(x,\tau, N,V^0) =O(\tau ), \quad \mathrm{as~~}\tau\rightarrow 0
$$
uniformly for all $x\in \R^n$, $N\in \mathbb N$ and $V^0$ as above.

\subsection{Comparison with earlier estimates}
McEneaney and Kluberg showed in~\cite[Thm 7.1]{MR2599910} that under Assumption~\ref{assum1}, for a given $V^0$,
\begin{align}\label{a-epsilon0}
 \epsilon_0(x,T,V^0)=O(\frac{1}{T}), \quad \mathrm{as~~}T\rightarrow +\infty
\end{align}
uniformly for all $x\in \R^n$. They also showed \cite[Thm 6.1]{MR2599910} that if in addition to Assumption~\ref{assum1}, the matrices
$\Sigma_m$ are all identical for $m\in \M$, then for a given $V^0$, 
\begin{align}\label{a-epsilonN}
\epsilon(x,\tau,N,V^0)=O(\sqrt{\tau}),\quad \mathrm{as~~}\tau\rightarrow 0
\end{align}
uniformly for all $x\in \R^n$ and $N\in \mathbb N$.
Their estimates imply that to get a sufficiently small approximation error $\epsilon$ we can use a horizon $T= O(1/\epsilon)$ and a discretization step $\tau =  O(\epsilon^2)$. Thus asymptotically the computational cost is:
$$
O(|\M|^{O(1/\epsilon^3)} n^3), \quad \mathrm{as~~} \epsilon \rightarrow 0.
$$
The same reasoning applied to our estimates shows a considerably smaller asymptotic growth rate of the computational cost (Corollary~\ref{coro-numberofiter}):
$$
O(|\M|^{O(-\log (\epsilon)/\epsilon)} n^3), \quad \mathrm{as~~} \epsilon \rightarrow 0
$$
McEneaney and Kluberg~\cite{MR2599910} gave a technically difficult proof of the estimates~\eqref{a-epsilon0} and~\eqref{a-epsilonN}, assuming that all the $\Sigma_m$'s are the same. They conjectured that the latter assumption can at least be released for a subclass of problems. This is supported by our results, showing that for the subclass of problems satisfying Assumption~\ref{assum2}, this assumption can be omitted. To this end, we use a totally different approach. Our main idea is to use the Thompson's metric to measure the error. It is well-known that the Thompson's metric defined on the space of positive-definite matrices is a Finsler metric and that the standard Riccati flow is
strictly contracting in Thompson's metric~(see~\cite{liveraniw, LawsonLimMR2338433}). However, we shall see that the Riccati equations appeared in the problem are indefinite and so we can not apply directly the contraction results.
It has been shown recently in~\cite{SGQThompsonrate} that the indefinite Riccati flows has a strict local contraction property in Thompson's metric under some technical assumptions. This local contraction result on the indefinite Riccati flow constitutes an essential part of our proofs. We shall also need an extension of the Thompson's metric to the space of supremum of quadratic functions.

Our approach derives a tighter estimate of $\epsilon_0(x,T,V^0)$ and $\epsilon(x,\tau,N,V^0)$ compared to previous results as well as an estimate of $\epsilon^\mathcal{P_{\tau}}(x,\tau, N,V^0)$ incorporating the pruning procedure. This new result justifies the use of pruning procedure of error $O(\tau^ 2)$ without increasing the asymptotic total approximation error order.

The paper is organized as follows. In Section~\ref{sec-maxpluspb}, we recall the switching linear quadratic control problem and the max-plus approximation method. In Section~\ref{sec-Thomp}, we recall the contraction results on the indefinite Riccati flow as well as the extension of the Thompson's metric to the space of supremum of quadratic functions. In Sections~\ref{sec-main} and~\ref{sec-theo2}, we present the estimates of the two approximation errors and part of the proofs. In Section~\ref{sec-tecprof}, we
show the proofs of some technical lemmas.
 Finally in Section~\ref{sec-dis}, we give some remarks and some numerical illustrations of the theoretical estimates.

\section{Problem statement}\label{sec-maxpluspb}
 We recall briefly the problem class and present some basic concepts and necessary assumptions. The reader can find more details in~\cite{curseofdim}.

\subsection{Problem class}\label{sec-probclass}
Let $\M=\{1,\cdots,M\}$ be a finite index set.  
We are interested in finding the value function of the following switching optimal control problem:
$$
V(x)= \sup_{w\in \W}\sup_{\mu\in \D_{\infty}} \sup_{T>0}\int_0^ T\frac{1}{2}(\xi_t^ {\mu_t})'D^{\mu_t} \xi_t^ {\mu_t}-\frac{\gamma^ 2}{2}|w_t|^ 2 dt
$$ 
where $$\D_{\infty}:=\{\mu:[0,\infty)\rightarrow \M,\mu\mathrm{~measurable}\},$$
$$\W:= \{w:[0,\infty)\rightarrow \R^ k:\int_{0}^ T|w_t|^ 2dt < \infty, \forall T<\infty\},$$ and $\xi$ is subject to:
\begin{align}\label{eqdfgg}
\dot \xi=A^ {\mu_t}\xi+\sigma^ {\mu_t}w_t, \enspace \xi_0=x.
\end{align}
In the sequel we denote $\Sigma^ m:=\frac{\gamma^ 2}{2}\sigma^ m (\sigma^ m)'$.
As in~\cite{curseofdim}, we make the following assumptions throughout the paper to guarantee the existence of $V$.
\begin{assump}\label{assum1}
~

\begin{itemize}
 \item There exists $c_A>0$ such that:
$$  
x'A^ mx\leq -c_A|x|^ 2,\quad \forall x\in \R^ n, m\in \M
$$
\item There exists $c_\sigma >0$ such that:
$$
|\sigma^ m |\leq c_\sigma,\quad \forall m\in \M
$$
\item All $D^ m$ are positive definite, symmetric, and there is $c_D$ such that:
$$
x'D^ mx\leq c_D|x|^ 2,\quad \forall x\in \R^ n, m\in \M,
$$
and
$$
c_A^ 2> \frac{c_Dc_\sigma^ 2}{\gamma^ 2}
$$
\end{itemize} 
\end{assump}

\subsection{Evolution semigroup}
Define the semigroup:
$$
S_T[V^0](x)=\sup_{w\in \W}\sup_{\mu\in \D_{T}} J(x,T;V^0;w,\mu)$$
where
$$
J(x,T;V^0;w,\mu):=\int_0^ T\frac{1}{2}\xi_t'D^{\mu_t} \xi_t-\frac{\gamma^ 2}{2} |w_t|^ 2 dt+V^0(\xi_T),
$$ 
$$\D_T=\{\mu: [0,T)\rightarrow \M, \mu\mathrm{~measurable}\},$$ and $\xi:[0,T]\rightarrow \R$ is subject to~\eqref{eqdfgg}.
For each $m\in \M$, define the semigroup $\{S^ m_t\}_{t\geq 0}$:
$$
S_T^m[V^0](x):= \sup_{w\in \W}J(x,T;V^0;w,\mu^m),
$$
where
$$
\mu^m_t=m,\quad t\in[0,T].
$$
Note that for each $t>0$, $S_t$ is \firstdef{max-plus linear} in the sense that for two functions $V_1,V_2:\R^ n\rightarrow \R$, 
we have:
\begin{align}
 &S_t[V_1+V_2](x)=S_t[V_1](x)+S_t[V_2](x),\quad \forall x\in \R^n\\\label{a-mplin}
&S_t[\sup\{V_1,V_2\}](x)=\sup\{S_t[V_1](x),S_t[V_2](x)\},\quad \forall x\in \R^ n
\end{align}

\subsection{Steady HJB equation}

For any $\delta\in (0,\gamma)$, define
\begin{align}\label{a-Gdelta}
G_{\delta}:=\{V\mathrm{~semiconvex~},V(x)\leq \frac{c_A(\gamma-\delta)^2}{c_\sigma^2}|x|^2, \enspace \forall x \}.
\end{align}
Then the value function $V$ is the unique viscosity solution of the following corresponding HJB PDE in the class $G_{\delta}$ for sufficiently small $\delta$~\cite{curseofdim}:
\begin{align}\label{a-define-H}
0=-H(x,\nabla V)=-\max_{m\in \M} H^ m(x,\nabla V).
\end{align}
where 
$$
H^ m(x,p)=(A^ mx)'p+\frac{1}{2}x'D^ mx+\frac{1}{2}p'\Sigma^ m p.
$$
It was shown in~\cite{curseofdim} that for $\delta$ sufficiently small and $V^0\in G_\delta$, 
\begin{align}\label{a-Tinfty}
\lim_{T\rightarrow \infty } S_T[V^0]=V
\end{align}
uniformly on compact sets. 

\subsection{Max-plus based approximation}\label{sec-method}
We review the basic steps of the algorithm proposed in~\cite{curseofdim} to approximate the value function $V$. Firstly using~\eqref{a-Tinfty} we are allowed to approximate $V$ by $S_T[V^0]$ for some sufficiently large $T$. We then choose a time-discretization step $\tau>0$ and a number of iterations $N$ such that $T=N\tau$ to approximate $S_T$ by
$$
S_T= \{S_\tau\}^N\simeq \{\tilde S_\tau\}^N
$$
where $\tilde S_\tau=\displaystyle\sup_{m\in \M} S_\tau^m$. 
At the end of $N$ iterations, we get our approximated value function represented by:
$$
V\simeq \sup_{i_1,\cdots,i_N}S_\tau^{i_N}\cdots S_\tau^{i_1}[V^0].
$$
If we choose $V^0(x)=\frac{1}{2}x'Px$ as a quadratic function, then the approximated value function will be the maximum of $|\M|^N$ quadratic functions. This so-called curse-of-complexity can be reduced by performing a pruning process at each iteration of the algorithm to remove some quadratic functions, see~\cite{a7}. 

\subsection{Approximation errors}
As pointed out in~\cite{MR2599910}, the approximation error comes from two parts. The first error source is
$$
\epsilon_0(x,T,V^0):=V(x)-S_{T}[V^0](x).
$$
This is due to the approximation of the infinite-horizon problem by a finite-horizon problem. The second source of error is
$$
\epsilon(x,\tau, N, V^0):= S_{N\tau}[V^0](x)-\{\tilde S_\tau\}^ N[V^0](x),
$$
due to the approximation of the semigroup by a time-discretization. In Section~\ref{sec-method} we mentioned that in practice
 a pruning procedure is needed so as to reduce 
the number of quadratic functions. More precisely, let $\{f_i\}_{i\in I}$ be a finite set of quadratic functions and 
\begin{align}\label{a-fsup}
f=\sup_{i\in I} f_i.
\end{align}
A pruning operation $\pP$ applied to $f$ produces an approximation of $f$ by selecting a subset $J\subset I$:
$$
f\simeq \pP[f]=\sup_{j\in J}f_j.
$$ 
If we take into account the pruning procedure, then the second error source should be written as:
\begin{align}\label{a-err-pru}
\epsilon^\mathcal{P_{\tau}}(x,\tau, N,V^0)= S_{N\tau}[V^0](x)-\{\pP_{\tau}\circ \tilde S_\tau\}^ N[V^0](x),
\end{align}
where $\pP_{\tau}$ represents a given pruning rule. We mark the subscript $\tau$ since it is expected that the pruning procedure be adapted with the time step $\tau$. In particular, we say that $\pP_\tau$ is a pruning procedure generating an error $O(\tau^{r})$ if there is $L>0$ such that for all function $f$ of the form~\eqref{a-fsup},
\begin{align}\label{a-pruningerror}
\pP_\tau[f]\leq f\leq (1+L\tau^{r}) \pP_\tau[f].
\end{align}
The special case without pruning procedure can be recovered by considering $r=+\infty$.


\section{Contraction properties of the indefinite Riccati flow}\label{sec-Thomp}
Before showing the main results, we present here an essential ingredient of our proof: the contraction properties of the indefinite Riccati flow.

\subsection{Loewner order and Thompson's part metric}
 We recall some basic notions and terminology. We refer the readers to~\cite{nussbaum88} for more
background.

We consider the space of $n$-dimensional symmetric matrices $\sym_n$ equipped with the operator norm $\|\cdot\|$. 
The space of positive semi-definite (resp. positive definite) matrices is denoted by $\sym_n^ +$ (resp. $\hat\sym_n^ +$ ).
The \firstdef{Loewner order}  "$\leq$" on $\sym_n$ is defined by:
$$
A\leq B \Leftrightarrow B-A\in \sym_n^ +.
$$
For $A\leq B$ we define the order intervals:
$$
\begin{array}{l}
[A,B]:=\{P\in \sym_n|A\leq P\leq B\}.
\end{array}
$$
For $P_1,P_2\in \hat \sym_n^+$, following~\cite{nussbaum88}, we define
$$
\begin{array}{l}
M(P_1/P_2):=\inf\{t>0:P_1\leq tP_2\}\\
\end{array}
$$
\begin{defi} The \firstdef{Thompson part metric} between two elements $P_1$ and $P_2$ of $\hat \sym_n^ +$ is
$$
d_T(P_1,P_2):=\log(\max \{M(P_1/P_2),M(P_2/P_1)\}).
$$
\end{defi}
\subsection{Contraction rate of the indefinite Riccati flow}
For each $m\in \M$, define the function $\Phi_m$: $\sym_n \rightarrow \sym_n$:
\begin{align}\label{a-phim}
\Phi_m(P)=(A^m)'P+PA^ m+P\Sigma^ mP+D^ m.
\end{align}
Associated to each $\Phi_m$ we define the flow map by:
$$M^m_t[P_0]=P(t),\quad t\in [0,T)$$
where $P(t):[0,T)\rightarrow \sym_n$ is a maximal solution of the following initial value problem on $\sym_n$:
\begin{align}\label{a-Riccati}
\dot P=\Phi_m(P), \quad P(0)=P_0.
\end{align}
When $V^0(x)=\frac{1}{2}x'P_0 x$ with $P_0\in \sym_n$, a classical result~\cite{YongZhoubook99} states that 
\begin{align}\label{a-S-M}
S_t^m[V^0](x)=\frac{1}{2}x'M^m_t[P_0]x, \quad t\in [0,T).
\end{align}
The standard Riccati equation refers to a vector field of the form~\eqref{a-phim} with $-\Sigma^ m$ and $D^ m$ positive semi-definite.
Here we are concerned with the \firstdef{indefinite} Riccati equation since the matrix coefficient $\Sigma^m$ is positive semi-definite.

The contraction property and the contraction rate calculus of the standard Riccati flow in Thompson's metric have been given in~\cite{liveraniw} and~\cite{LawsonLimMR2338433}. However, their approach depends on the algebraic property of the associated symplectic operator, which fails in the indefinite case. In~\cite{SGQThompsonrate}, the authors give a general explicit formula for the local contraction rate of a flow, in Thompson's metric, from which it follows that under additional constraints on the matrix coefficients, the Riccati flow is still a local contraction in the indefinite case. Below is the additional assumption needed to apply this new contraction result:
\begin{assump}\label{assum2}
 There is $m_D>0$ such that 
$$
x'D^ mx\geq m_D|x|^2,\quad \forall x\in \R^n,m\in \M
$$
and $$
\frac{c_\sigma^2}{\gamma^2}m_D>(c_A-\sqrt{c_A^2-c_Dc_\sigma^2/\gamma^2})^2.
$$
\end{assump}
In the sequel we denote
$$
\lambda_1=\frac{\gamma^2(c_A-\sqrt{c_A^2-c_D c_\sigma^ 2/\gamma^ 2})}{c_\sigma^2} ,\enspace\lambda_2:=\sqrt{m_D\gamma^ 2/c_\sigma^2}.
$$
\begin{rem}\label{rem-inv}
Under Assumption~\ref{assum2}, we can choose $\epsilon>0$ sufficiently small so that 
\begin{align}\label{a-posi}
M_{t_0}^m[0]\geq \epsilon I, \enspace \mathrm{~for~some~}t_0>0, m\in \M.
\end{align}
Since $\Phi_m(0)=D^m \geq m_D I$ for all $m\in \M$, we can let $\epsilon$ be sufficiently small such that $\Phi_m(\epsilon I)\geq 0$ for all $m\in \M$. Besides, for any $\lambda \in [\lambda_1,\lambda_2)$, we have $\Phi_m(\lambda I)\leq 0$ for all $m\in \M$. Then it follows from a standard result on the Riccati equation that: 
\begin{align}\label{a-invariantset}
M_t^m[P_0]\in [\epsilon I,\lambda I], \enspace \forall m\in \M, t\geq 0, P_0 \in[\epsilon I,\lambda I].
\end{align}
\end{rem}
The main ingredient to make our proofs is the following theorem:
\begin{theo}[Corollary 4.6 in~\cite{SGQThompsonrate}]\label{theo-riccaticontraction}
Under Assumptions~\ref{assum1} and~\ref{assum2}, for any 
$\lambda \in [\lambda_1,\lambda_2)$, there is $\alpha>0$ such that for all $P_1,P_2\in (0,\lambda I]$,
$$
d_T(M_t^m[P_1],M_t^m[P_2])\leq e^{-\alpha t}d_T(P_1,P_2),\enspace \forall t\geq 0, m\in \M.
$$
\end{theo}

\subsection{Extension of the contraction result to the space of functions}
Now we extend the definition of Thompson's metric to the space of non-negative functions.
For two functions $f,g:\R^ n\rightarrow \R$, we consider the standard partial order "$\leq$" by:
$$
f\leq g \Leftrightarrow f(x)\leq g(x),\enspace \forall x\in \R^ n,
$$
which coincides with the Loewner order on the set of quadratic forms.
Similarly, for  $f, g:\R^n\rightarrow \R_+$ we define
$$
M(f/g):=\inf\{t>0:f\leq t g\}
$$
We say that $f$ and $g$ are comparable if $M(f/g)$ and $M(g/f)$ are finite. In that case, we can define the \firstdef{"Thompson metric"} between $f, g:\R^n\rightarrow \R_+$ by:
\begin{align}\label{a-thompfunctions}
d_T(f,g)=\log(\max\{M(f/g),M(g/f)\}).
\end{align}
Then the following lemma can be easily proved using the definition:
\begin{lemma}\label{lemma-fig}
 Let $f,g:\R^n\rightarrow \R_+$ be given by pointwise maxima of non-negative functions
$$
f:=\sup_{i\in I} f_i,\enspace g:=\sup_{i\in I} g_i
$$
Then
\begin{align}\label{a-fg}
 d_T(f,g)\leq \sup_{i\in I} d_T(f_i,g_i).
\end{align}
\end{lemma}

The following result is a consequence of the order-preserving character of the Riccati flow and of the contraction property in Theorem~\ref{theo-riccaticontraction}.
\begin{lemma}\label{coro-contractioncoro}
 Under Assumptions~\ref{assum1} and~\ref{assum2}, let
$\lambda \in [\lambda_1,\lambda_2)$ and $\epsilon>0$ such that~\eqref{a-invariantset} holds. Then there is $\alpha>0$ such that for any two functions $V_1$ and $V_2$ of the form:
$$
V_1(x)=\sup_{j\in J}\frac{1}{2}x'P_jx,\qquad V_2(x)=\frac{1}{2}x'Qx,\enspace 
$$
where $Q,P_j\in [\epsilon I,\lambda I]$ for all $j\in J$, we have
$$
d_T(S_{t/N}^{i_N}\cdots S^{i_1}_{t/N}[V_1],S_{t/N}^{i_N}\cdots S^{i_1}_{t/N}[V_2])\leq e^ {-\alpha t}\log(\frac{\lambda}{\epsilon})
$$
for all $t\geq 0$, $N\in \mathbb N$ and $(i_1, \cdots,i_N)\in \M^ N$.
\end{lemma}
\begin{proof} 
For all $P, Q \in [\epsilon I,\lambda I]$, by~\eqref{a-invariantset} and Theorem~\ref{theo-riccaticontraction} we have
$$ d_T(M_{t/N}^{i_N}\cdots M^{i_1}_{t/N}[P],M_{t/N}^{i_N}\cdots M^{i_1}_{t/N}[Q])\leq e^{-\alpha t}d_T(P,Q)$$for all $t\geq 0$, $N\in \mathbb N$ and $(i_1, \cdots,i_N)\in \M^ N$. 
Now by the max-plus linearity of the semigroup~\eqref{a-mplin}, Lemma~\ref{lemma-fig} and the relationship between the semigroup and the flow~\eqref{a-S-M}, we get
$$\begin{array}{l} d_T(S_{t/N}^{i_N}\cdots S^{i_1}_{t/N}[V_1],S_{t/N}^{i_N}\cdots S^{i_1}_{t/N}[V_2])\\ 
\leq\displaystyle\sup_{j\in J} d_T(M_{t/N}^{i_N}\cdots M^{i_1}_{t/N}[P_j],M_{t/N}^{i_N}\cdots M^{i_1}_{t/N}[Q])\\\leq e^{-\alpha t}\displaystyle\sup_{j\in J} d_T(P_j, Q)\leq e^{-\alpha t}\log(\frac{\lambda}{\epsilon}).\end{array}$$\end{proof}

\section{Finite horizon error estimate}\label{sec-main}

We first study the finite horizon truncation error $\epsilon_0(x,T,V^0)$. Below is one of our main results:

\begin{theo}\label{theo-error1}
 Under Assumptions~\ref{assum1} and~\ref{assum2}, let
$\lambda \in [\lambda_1,\lambda_2)$ and $\epsilon>0$ such that~\eqref{a-posi} and~\eqref{a-invariantset} hold. There exist $\alpha>0$ and $K>0$ such that,
$$
\epsilon_0(x,T,V^0)\leq K e^ {-\alpha T}|x|^ 2,\enspace \forall x,
$$
for all $T>0$ and $V^0(x)=\frac{1}{2}x'P_0x$ with $P_0\in [\epsilon I,\lambda I]$.
\end{theo}

The remaining part of the section is devoted to the proof of the above theorem.
We shall need the following technical lemma. The proof is deferred to Section~\ref{sec-pflSt}.
\begin{lemma}[Approximation by piecewise constant controls]\label{lemma-St}
Let $V^0: \R^n\rightarrow \R$ be a given locally Lipschitz function. For any $T>0$
we have
$$
S_T[V^0]=\sup_{N}\sup_{i_1,\cdots i_N}S_{T/N}^{i_N}\cdots S^{i_1}_{T/N}[V^0].
$$
\end{lemma}
From now on we make Assumptions~\ref{assum1} and~\ref{assum2}. We also fix 
$\lambda \in [\lambda_1,\lambda_2)$ and $\epsilon>0$ satisfying \eqref{a-posi} and~\eqref{a-invariantset}.

\begin{rem}
 Since the interval $[\epsilon I,  \lambda I]$ is invariant by any operator $\{S_\tau^m\}_{\tau\geq 0,m\in \M}$, it is direct from Lemma~\ref{lemma-St} that 
\begin{align}\label{a-bound}
\frac{\epsilon}{2} |x|^2\leq S_T[V^0](x) \leq \frac{\lambda}{2} |x|^2,\enspace \forall T>0
\end{align} 
for 
all $V^0(x)=\frac{1}{2}x'Px$ with $P\in [\epsilon I,\lambda I]$.
\end{rem}
\begin{coro}\label{coro-V}
  The value function $V$ is a pointwise supremum of quadratic functions 
$$
V(x)=\sup_{j\in J} \frac{1}{2}x'P_j x
$$
where $P_j\in [\epsilon I, \lambda I]$ for all $j\in J$.
\end{coro}
\begin{proof}
 By definition, we have:
$$
V(x)=\sup_{T>0} S_T[0](x),\enspace \forall  x.
$$
By~\eqref{a-posi}, there is $t_0>0$ and $m\in \M$ such that 
$$
M_{t_0}^m[0]\geq \epsilon I.
$$
Besides, by the monotonicity of the semigroup,
$$
 S_T[S_{t_0}^m[0]](x)\leq  S_T[S_{t_0}[0]](x),\enspace \forall x, T>0
$$
and
$$
 S_T[0](x)\leq  S_T[S_{t_0}^m[0]](x),\enspace \forall x,T>0.
$$
Since
$$
V(x)=\sup_{T>0} S_T[0](x)=\sup_{T}S_T[S_{t_0}[0]](x),\enspace \forall x,
$$
we get that:
$$
\sup_{T}S_T[S_{t_0}^m[0]](x)\leq V(x)\leq \sup_{T}S_T[S_{t_0}^m[0]](x),\enspace x.
$$
Hence by Lemma~\ref{lemma-St}:
$$
V(x)=\sup_{T} S_T[S_{t_0}^m[0]]=\sup_{T} \sup_{N}\sup_{i_1,\cdots i_N}S_{T/N}^{i_N}\cdots S^{i_1}_{T/N} S_{t_0}^m[0].
$$
Now using the invariance of the interval $[\epsilon I,\lambda I]$ in ~\eqref{a-invariantset}, we know that
$$
M_{T/N}^{i_N}\cdots M^{i_1}_{T/N} M_{t_0}^m[0]\in [\epsilon I,\lambda I],
$$
for all $T>0$, $N\in \mathbb N$ and $i_1,\dots,i_N\in \M$.
Consequently $V$ is a pointwise maximum of quadratic functions $\frac{1}{2}x'P_jx$ with $P_j\in [\epsilon I,\lambda I]$.
\end{proof}

Using the above lemma we show that:
\begin{prop}\label{prop-th1}
 There is $\alpha>0$ such that for all $V^0(x)=\frac{1}{2}x'P_0 x$ with $P_0\in [\epsilon I,\lambda I]$,
$$
d_T(V,S_T[V^0])\leq e^ {-\alpha T}\log(\frac{\lambda}{\epsilon}),\enspace \forall T>0.
$$
\end{prop}
\begin{proof}
By Corollary~\ref{coro-V}, the value function $V$ is a pointwise supremum of quadratic functions:
$$
V(x)=\sup_{j\in J} \frac{1}{2}x'P_j x
$$
where $P_j \in [\epsilon I, \lambda I]$ for all $j\in J$. Let any $V^0(x)=\frac{1}{2}x'P_0 x$ with $P_0\in [\epsilon I,\lambda I]$.
By Corollary~\ref{coro-contractioncoro},  we have:
$$
d_T(S_{T/N}^{i_N}\cdots S^{i_1}_{T/N}[V],S_{T/N}^{i_N}\cdots S^{i_1}_{T/N}[V^0])\leq e^ {-\alpha T}\log(\frac{\lambda}{\epsilon})
$$
for all $T\geq 0$, $N\in \mathbb N$ and $(i_1, \cdots,i_N)\in \M^ N$.
We also know from Lemma~\ref{lemma-St} that
$$
V=S_T[V]=\sup_{N}\sup_{i_1,\dots,i_N} S_{T/N}^{i_N}\dots S_{T/N}^{i_1}[V],
$$
and that
$$
S_T[V^0]=\sup_{N}\sup_{i_1,\dots,i_N} S_{T/N}^{i_N}\dots S_{T/N}^{i_1}[V^0].
$$
Therefore by Lemma~\ref{lemma-fig},
$$
\begin{array}{l}
d_T(V,S_T[V^0])=d_T(S_T[V],S_T[V^0])\\
\leq \displaystyle\sup_{N}\sup_{i_1,\cdots i_N} d_T(S_{T/N}^{i_N}\cdots S^{i_1}_{T/N}[V],S_{T/N}^{i_N}\cdots S^{i_1}_{T/N}[V^0])\\
 \leq e^ {-\alpha T} \log(\frac{\lambda}{\epsilon}).
\end{array}
$$
\end{proof}
Now we have all the necessary elements to prove Theorem~\ref{theo-error1}.
\begin{proof}[Proof of Theorem~\ref{theo-error1}]
  Let any $V^0(x)=\frac{1}{2}x'P_0 x$ with $P_0\in [\epsilon I,\lambda I]$. By Proposition~\ref{prop-th1} and~\eqref{a-thompfunctions}, there 
is $\alpha>0$ such that
$$
V(x)\leq e^{e^{-\alpha T}\log(\lambda/\epsilon)} S_T[V^0](x),\enspace \forall T>0, x\in \R^n.
$$
Thus there is constant $L>0$ such that
$$
V(x)\leq (1+L e^{-\alpha T}) S_T[V^0](x),\enspace \forall T>0, x\in \R^n
$$
This leads to
$$
\epsilon_0(x,T,V^0)\leq  L e^{-\alpha T} S_T[V^0](x)\leq \frac{\lambda L}{2} e^{-\alpha T}|x|^2,\enspace \forall T>0, x\in \R^n.
$$
where the last inequality follows from~\eqref{a-bound}. It is clear that the constant $K=\frac{\lambda L}{2}$ is independent of $P_0\in[\epsilon I,\lambda I]$.
\end{proof}

\section{Discrete-time approximation error estimate}\label{sec-theo2}
In this section we analyze the discrete-time approximation error $\epsilon^ {\pP_\tau}(x,\tau, N, V^0) $. Our main result is:
\begin{theo}\label{theo-error2}
Let $r>1$. Suppose that for each $\tau>0$ the pruning operation $\pP_\tau$ generates an error $O(\tau^ {r})$ (see~\eqref{a-pruningerror}).
Under Assumptions~\ref{assum1} and~\ref{assum2}, let
$\lambda \in [\lambda_1,\lambda_2)$ and $\epsilon>0$ such that~\eqref{a-invariantset} holds. Then there exist $\tau_0>0$ and $L>0$ such that
$$ \epsilon^ {\pP_\tau}(x,\tau, N, V^0)
\leq L\tau^ {\min \{1,r-1\}}|x|^ 2,\enspace \forall x,
$$
for all $N\in \mathbb N$, $\tau \leq \tau_0$ and $V^0(x)=\frac{1}{2}x'P_0x$ with $P_0\in [\epsilon I, \lambda I]$.
\end{theo}
The remaining part of the section is devoted to the proof of Theorem~\ref{theo-error2}.
We first state a technical lemma which is proved  in Section~\ref{sec-pflone}.

\begin{lemma}\label{l-onestep}
 Let $\mathcal K\subset S_n$ be a compact convex subset. There exist $\tau_0>0$ and $L>0$ such that
$$
S_\tau[V^0](x)\leq \tilde S_\tau[V^0](x)+L\tau^ 2 |x|^2, \enspace \forall x,
$$
for all $\tau \in[0,\tau_0]$ and $V^0(x)=\frac{1}{2}x'P_0x$ with $P_0\in \mathcal K$.
\end{lemma}

Now we take into account the pruning procedure and analyze the error of the following approximation
$$
S_\tau \simeq \pP_\tau \circ \tilde S_\tau.
$$
Below is a direct consequence of Lemma~\ref{l-onestep} and~\eqref{a-invariantset}.
\begin{coro}\label{coro-1}
 Let $\epsilon$, $\lambda$, $r$ and $\pP_\tau$ be as in~Theorem~\ref{theo-error2}. Then there exist $\tau_0>0$ and $L>0$ such that:
$$
S_\tau[V^0](x)\leq  (1+L\tau^{\min\{2,r\}})\pP_\tau \circ \tilde S_\tau[V^0](x),\enspace \forall x,
$$
for all $\tau \in[0,\tau_0]$ and $V^0(x)=\frac{1}{2}x'P_0x$ with $P_0\in [\epsilon I, \lambda I]$.
\end{coro}

We are ready to give a proof of Theorem~\ref{theo-error2}:
\begin{proof}[Proof of Theorem~\ref{theo-error2}]
Denote $s=\min\{2,r\}$.
Let any $\lambda'>0$ such that
$$
\lambda <\lambda' < \lambda_2.
$$
Denote $\delta=\lambda'/\lambda$. Consider the two compact convex subsets $\K_0=[\epsilon I,\lambda I]$ and $\K_1=[\epsilon I,\lambda' I]$.
It is easily verified that:
$$  \Phi_m(\lambda' I)\leq 0,\enspace \forall m\in \M.$$Therefore for all $P_0\in \K_0$, $P_1\in \K_1$, $t\geq 0$ and $m\in \M$,
\begin{align}\label{C2}
M_t^m[P_0] \in \K_0,\quad
M_t^m[P_1]\in \K_1.
\end{align}
 By Corollary~\ref{coro-1},  there is $\tau_0$ and $L>0$ such that for all $\tau\in [0,\tau_0]$ and $V^0=\frac{1}{2}x'P x$ with $P\in \K_1$:
\begin{align}\label{afdfsdsf}
S_\tau[V^0]\leq (1+L\tau^{s}) \pP_\tau\circ \tilde S_\tau[V^0].
\end{align}
Let $\tau_0>0$ be sufficiently small such that:
$$
(1+L\tau^ {s})^ {\frac{1}{1-e^ {-\alpha \tau}}}\leq \delta, \quad \forall \tau\in [0,\tau_0].
$$
Let any $V^0(x)=\frac{1}{2}x'P_0x$ with $P_0\in \K_0$ and $\tau \in [0,\tau_0]$, we are going to prove by induction on $N\in \mathbb{N}$ the following inequalities:
$$
S_{N\tau}[V^0]\leq (1+L\tau^{s})^ {1+e^ {-\alpha \tau}+\cdots+e^{-(N-1)\alpha \tau}} \{\pP_\tau \circ \tilde S_\tau\}^ N[V^0] ,\quad \forall N\in \mathbb{N}.
$$
The case $N=1$ is already given in~\eqref{afdfsdsf}. Suppose that the above inequality is true for some $k\in \mathbb{N}$, that is,
$$
S_{k\tau}[V^0]\leq L_k\{\pP_\tau\circ \sup_m S^ m_\tau\}^ k[V^ 0]
$$
where $L_k=(1+L\tau^ {s})^ {1+e^ {-\alpha \tau}+\cdots+e^{-(k-1)\alpha \tau}}$.
We denote by $I_k \subset \M^k$ the subset such that
$$
\{\pP_\tau\circ \sup_m S^ m_\tau\}^ k[V^ 0]=\sup_{(i_1,\cdots,i_k)\in I_k} S_\tau^{i_k}\cdots S_\tau^{i_1} [V^0].
$$
Thus,
\begin{align}\label{a-k}
S_{k\tau}[V^0] \leq \sup_{(i_1,\cdots,i_k)\in I_k} L_k S_\tau^{i_k}\cdots S_\tau^{i_1} [V^0].
\end{align}
From \eqref{C2}, we know that for all $(i_1,\cdots,i_k)\in I_k$
\begin{align}\label{a-K1}
M_ \tau^{i_k}\cdots M_\tau^{i_1} [P_0]\in \K_0.
\end{align}
Besides,
$$
1\leq L_k\leq (1+L\tau^{s})^ {\frac{1}{1-e^ {-\alpha \tau}}}\leq \delta.
$$
Thus for all $(i_1,\cdots,i_k)\in I_k$,
\begin{align}\label{a-K0}
L_k (M_ \tau^{i_k}\cdots M_\tau^{i_1} [P_0])\in \K_1
\end{align}
Recall that
$$
L_k S_\tau^{i_k}\cdots S_\tau^{i_1} [V^0](x)=\frac{L_k}{2}x'(M_ \tau^{i_k}\cdots M_\tau^{i_1} [P_0]) x,$$
then by applying~\eqref{a-k} and~\eqref{afdfsdsf}, we obtain that
\begin{align}\label{a-iteration}
\begin{array}{l}
S_{(k+1)\tau}[V^0]=S_{\tau}[S_{k\tau}[V^0]]\\
\leq \displaystyle\sup_{(i_1,\cdots,i_k)\in I_k} S_{\tau}[L_k S_\tau^{i_k}\cdots S_\tau^{i_1} [V^0]]\\
\leq \displaystyle\sup_{(i_1,\cdots,i_k)\in I_k} (1+L\tau^ {s}) \pP_\tau\circ \tilde S_\tau[ [L_k S_\tau^{i_k}\cdots S_\tau^{i_1} [V^0]]]\\
\end{array}
\end{align}
Now by Theorem~\ref{theo-riccaticontraction}, there is $\alpha>0$ such that for all $P_1,P_2\in \K_1$ and $m\in \M$
$$
d_T(M_\tau^m[P_1],M_\tau^m[P_2])\leq e^{-\alpha \tau} d_T(P_1,P_2)
$$
Therefore from~\eqref{a-K1} and~\eqref{a-K0} we get that for any $(i_1,\dots,i_k)\in I_k$ and $m\in \M$
$$
\begin{array}{ll}
&d_T(M_\tau^mM_ \tau^{i_k}\cdots M_\tau^{i_1} [P_0], M_\tau^m[L_k(M_ \tau^{i_k}\cdots M_\tau^{i_1} [P_0])])\\
&\leq e^ {-\alpha \tau} d_T(M_ \tau^{i_k}\cdots M_\tau^{i_1} [P_0],L_k(M_ \tau^{i_k}\cdots M_\tau^{i_1} [P_0]))=e^ {-\alpha \tau}\log L_k.
\end{array}
$$
 This implies that
$$
M_\tau^m[L_k(M_ \tau^{i_k}\cdots M_\tau^{i_1} [P_0])]\leq L_k^ {e^ {-\alpha \tau}} M_\tau^mM_\tau^{i_k}\cdots M_\tau^{i_1}[P_0], \enspace \forall m\in \M
$$
which is,
$$
S_{\tau}^m [L_k S_\tau^{i_k}\cdots S_\tau^{i_1}[V^0]]\leq L_k^ {e^ {-\alpha \tau}} S_{\tau}^m S_\tau^{i_k}\cdots S_\tau^{i_1}[V^0], \enspace \forall m\in \M.
$$
Therefore we deduce from the inequality~\eqref{a-iteration}:
$$
\begin{array}{ll}
S_{(k+1)\tau}[V^0]
&\leq (1+L\tau^{s} ) \pP_\tau[\displaystyle\sup_{m\in \M,(i_1,\cdots,i_k)\in I_k} S_{\tau}^m [L_k S_\tau^{i_k}\cdots S_\tau^{i_1}[V^0]]]
\\
&\leq (1+L\tau^ {s} )L_k^ {e^ {-\alpha \tau}} \pP_\tau[\displaystyle\sup_{m\in \M,(i_1,\cdots,i_k)\in I_k}  S_{\tau}^m S_\tau^{i_k}\cdots S_\tau^{i_1}[V^0]]\\
&= (1+L\tau^ {s} )^ {1+e^ {-\alpha \tau}+\cdots+e^{-k\alpha \tau}} \{\pP_\tau\circ \tilde S_\tau\}^ {k+1}[V^0].
\end{array}
$$
Thereby we proved that
$$
S_{N\tau}[V^0]\leq (1+L\tau^ {s} )^ {\frac{1}{1-e^ {-\alpha \tau}}} \{\pP\circ \tilde S_\tau\}^ N [V^0],\quad \forall N\in \mathbb{N}.
$$
 Note that
$$
\lim_{\tau\rightarrow 0^+} \frac{(1+L\tau^ {s} )^ {\frac{1}{1-e^ {-\alpha \tau}}}-1}{\tau^ {s-1}}=\frac{L}{\alpha},
$$
from which we deduce the existence of $\tau_0$ and $K>0$ such that for all $\tau \in[0,\tau_0]$, $N\in \mathbb{N}$ and $V^0(x)=\frac{1}{2}x'Px$ with $P\in [\epsilon I,\lambda I]$
$$
\{S_\tau\}^N[V^0]\leq (1+K\tau ^ {s-1})\{\pP\circ \tilde S_\tau\}^N[V^0].
$$
which leads to 
$$
\epsilon^ {\pP_\tau}(x,\tau, N, V^0) \leq K \tau^{\min\{1,r-1\}} \{\pP\circ \tilde S_\tau\}^N[V^0] \leq \frac{K\lambda }{2} \tau^{\min\{1,r-1\}}  |x|^2.
$$
\end{proof}

\begin{rem}\label{rem-crucial}
 It should be pointed out that the crucial point is having $\alpha>0$. If this is not the case ($\alpha=0$), then the iteration~\eqref{a-iteration} only leads to:
$$
\begin{array}{l}
d_T(S_{N\tau}[V^0],\{\pP_\tau \circ \tilde S_\tau\}^ {N}[V^0])\leq LN\tau^ {s},\enspace \forall N\in \mathbb N.
\end{array}
$$
\end{rem}

\section{Proofs of the technical lemmas}\label{sec-tecprof}

\subsection{Proof of Lemma~\ref{lemma-St}}\label{sec-pflSt}
For two functions $\mu,\nu\in \D_T$ we consider the metric $d(\mu,\nu)$ defined by the measure of subset on which the two controls $\mu$ and $\nu$ differ from each other:
\begin{align}\label{a-dmunu}
d(\mu,\nu)=\int_{0}^T 1_{\mu\neq \nu} dt.
\end{align}
The proof of Lemma~\ref{lemma-St} needs the next lemma. 
It shows that 
the objective function is continued on the variable $\mu\in \D_T$ with respect to the metric $d$ defined in~\eqref{a-dmunu}.
\begin{lemma}\label{anapro}
Let $V^ 0:\R^n\rightarrow \R$ be a locally Lipschitz function. Let $x\in \R^ n$ and $T>0$.
 Given $\mu\in \D_T$ and $w\in \W_T$, for any $\epsilon>0$, there is $\delta_0>0$ such that 
$$
|J(x,T;V^ 0;\mu,w)-J(x,T;V^ 0;\tilde \mu,w)|\leq \epsilon,
$$
for all $\tilde \mu\in \D_T$ such that $d(\mu,\tilde \mu)\leq \delta_0$. 
\end{lemma}
\begin{proof}
 Let any $\tilde \mu \in \D_T$ and denote:
$$
\delta=d(\mu,\tilde \mu).
$$
Let $\xi$ and $\tilde \xi$ be respectively the solutions to~\eqref{eqdfgg} under the control $(\mu,w)$ and $(\tilde \mu,w)$. Thus
$$
\xi_t-\tilde \xi_t=\int_0^t A^ {\mu_s}\xi_s+\sigma^ {\mu_s}w_s-(A^ {\tilde \mu_s}\tilde \xi_s+\sigma^ {\tilde \mu_s}w_s) ds,\enspace \forall t\in [0,T].
$$
Denote
 $$L=\displaystyle\max (\max_m \|A^m\|,\max_m |\sigma^m|,\max_m \|D^m\|,(\int_0^T(|\xi_s|+|w_s|)^2 ds)^{1/2}).$$
We have:
$$
\begin{array}{ll}
|\xi_t-\tilde \xi_t| &\leq \int_0^t |A^{\mu_s}\xi_s-A^{\tilde \mu_s}\xi_s|+|A^{\tilde \mu_s}\xi_s-A^{\tilde \mu_s}\tilde \xi_s|+|\sigma^{\mu_s}-\sigma^{\tilde \mu_s}||w_s|ds
\\&\leq \int_0^t L |\xi_s-\tilde \xi_s | ds+\int_0^t 1_{\mu\neq \tilde \mu}(\|A^{\mu_s}-A^{\tilde \mu_s}\||\xi_s|+|\sigma^{\mu_s}-\sigma^{\tilde \mu_s}||w_s|) ds\\
&\leq \int_0^t L|\xi_s-\tilde \xi_s| ds+2L\int_0^t  1_{\mu\neq \tilde \mu}(|\xi_s|+|w_s|)ds \\
&\leq \int_0^t L|\xi_s-\tilde \xi_s| ds+2L(\int_0^t 1_{\mu\neq \tilde \mu} ds)^{1/2}(\int_0^t (|\xi_s|+|w_s|)^2)^{1/2}\\
&\leq \int_0^t L|\xi_s-\tilde \xi_s| ds+2L^2\delta^{\frac{1}{2}},\qquad \qquad \qquad \forall t\in [0,T].
\end{array}
$$
By Gronwall's Lemma,
$$
|\xi_t-\tilde \xi_t|\leq 2L^2 \delta^{\frac{1}{2}}e^{Lt}\leq  L \delta^{\frac{1}{2}},\quad \forall t\in [0,T].
$$
Then
$$
|\tilde \xi_t|\leq \sup_{t\in[0,T]}|\xi_t|+L\delta ^{\frac{1}{2}}\leq L,\forall t\in [0,T].
$$ Note that $L$ is independent of $\tilde \mu$.
Now by the local Lipschitz property of $V^0$ and the boundedness of $\xi$ and $\tilde \xi$, there is $L>0$ such that:
$$
|V^0(\xi_T)-V^0(\tilde \xi_T)|\leq L |\xi_T-\tilde \xi_T|\leq L \delta^{\frac{1}{2}}
$$
Besides,
$$
\begin{array}{l}
|\int_0^T \xi_t' D^{\mu_t} \xi_t-\tilde \xi_t' D^{\tilde \mu_t} \tilde \xi_t dt|\\
\leq \int_0^T |\xi_t'D^{\mu_t}(\xi_t-\tilde \xi_t)|+|\tilde \xi_t' D^{\mu_t}(\xi_t-\tilde \xi_t)|+|\tilde \xi_t'(D^{\mu_t}-D^{\tilde \mu_t})\tilde \xi_t|dt\\
\leq L\int_0^T (|\xi_t-\tilde \xi_t|+ 1_{\mu\neq \tilde \mu} ) dt\\
\leq L(\delta ^{\frac{1}{2}}+\delta)
\end{array}
$$
Thus there is a constant $L$ independent of $\tilde \mu$ such that:
$$
|J(x,T;V^ 0;\mu,w)-J(x,T;V^ 0;\tilde \mu,w)|\leq L(\delta^{\frac{1}{2}}+\delta)
$$
whence for any $\epsilon>0$ there is $\delta_0>0$ such that 
$$
|J(x,T;V^ 0;\mu,w)-J(x,T;V^ 0;\tilde \mu,w)|\leq \epsilon
$$
for all $\tilde \mu \in \D_T$ such that $d(\mu,\tilde \mu)\leq \delta_0$.
\end{proof}

Using this, we can prove Lemma~\ref{lemma-St}:
\begin{proof}[Proof of Lemma~\ref{lemma-St}]
 Let $V^ 0$ be a locally Lipschitz function. Fix $x\in \R^n$. Let $\mu \in \D_\tau$ and $w\in \W_\tau$ be $\frac{\epsilon}{2}$-optimal for $S_\tau[V^0](x)$, that is:
\begin{align}\label{artdg}
S_\tau[V^0](x)\leq J(x,\tau;V^ 0;\mu,w)+\frac{\epsilon}{2}.
\end{align}
By Lemma~\ref{anapro}, there is $\delta_0>0$ such that:
\begin{align}\label{affghss}
|J(x,\tau;V^ 0;\tilde \mu,w)-J(x,\tau;V^0;\mu, w)|\leq \frac{\epsilon}{2}
\end{align}
for all $\tilde \mu \in \D_\tau$ such that $d(\mu,\tilde \mu)\leq \delta_0$. 
Now it remains to prove that there is at least one piecewise constant function $\tilde \mu \in \D_\tau$
such that $d(\mu,\tilde \mu)\leq \delta_0$.
To this end, by Lusin's theorem~\cite{MR1681462}, there is a compact $K\subset [0,\tau]$ such that
                                                               $$
\int_0^\tau 1_{K}>\tau-\delta_0
$$
                     and                                     
 the restriction of $\mu$ on $K$ is continuous, thus uniformly continuous. Let $\delta>0$ such that for all $t,s\in K$ and $|t-s|\leq \delta$,
$$
|\mu(t)-\mu(s)|\leq \frac{1}{2}
$$
which implies
$$
\mu(t)=\mu(s).
$$
Now let $N_0\in \mathbb{N}$ such that $\frac{1}{N_0}<\delta$. We construct a piecewise constant function $\tilde \mu\in \D_\tau$ as following. For $i\in \{0,1\cdots,N_0-1\}$, let
$$
\tilde \mu(\frac{i}{N_0}\tau)=\left\{\begin{array}{ll}\mu(s),& \textrm{if~ there~is~}  s\in K\cap [\frac{i}{N_0}\tau,\frac{i+1}{N_0}\tau)\\
1 ,&\textrm{else}
\end{array}\right.
$$
and
$$
\tilde \mu(t)=\tilde \mu(\frac{i}{N_0}\tau), \quad t\in [\frac{i}{N_0}\tau,\frac{i+1}{N_0}\tau).
$$
Since $\mu(s)=\mu(t)$ for all $s,t\in K\cap [\frac{i}{N_0}\tau,\frac{i+1}{N_0}\tau)$, it follows that 
$$
\mu(t)=\tilde \mu(t),\quad \forall t\in K.
$$
Thus
$$
\int_0^\tau 1_{\mu \neq \tilde \mu} dt \leq \int_0^\tau 1- 1_{ K} dt \leq \delta_0.
$$
So $d(\mu,\tilde \mu)\leq \delta_0$ and $\tilde \mu$ is constant on interval 
$[\frac{i}{N_0}\tau,\frac{i+1}{N_0}\tau)$ for all $i\in \{0,1,\cdots,N_0-1\}$. Hence, by~\eqref{affghss},
$$
J(x,\tau;V^ 0;\mu,w)\leq J(x,\tau;V^ 0;\tilde \mu,w)+\frac{\epsilon}{2}\leq \sup_{i_1,\cdots,i_{N_0}}
S_{\tau/N_0}^{i_{N_0}}\cdots S_{\tau/N_0}^{i_1}[V^0](x)+\frac{\epsilon}{2}
$$
Now by~\eqref{artdg}, we get
$$
S_\tau[V^0](x)\leq \sup_{i_1,\cdots,i_{N_0}}S_{\tau/N_0}^{i_{N_0}}\cdots S_ {\tau/N_0}^{i_1}[V^0](x)+\epsilon\leq \sup_{N}\sup_{i_1,\cdots,i_N} S_{\tau/N}^{i_{N}}\cdots S_ {\tau/N}^{i_1}[V^0](x)+\epsilon.
$$
This is true for any $\epsilon>0$, we conclude that:
$$
S_\tau[V^0](x)=\sup_{N}\sup_{i_1,\cdots,i_N} S_{\tau/N}^{i_{N}}\cdots S_{\tau/N}^{i_1}[V^0](x)
$$
for all $x\in \R^n$. Thus
$$
S_\tau[V^0]=\sup_{N}\sup_{i_1,\cdots,i_N} S_{\tau/N}^{i_{N}}\cdots S_{\tau/N}^{i_1}[V^0].
$$
\end{proof}

\subsection{Proof of Lemma~\ref{l-onestep}}\label{sec-pflone}
The proof of Lemma~\ref{l-onestep} shall need the following estimates:
\begin{lemma}\label{l-TaylorD}
 Let $\mathcal K\subset \sym_n$ be a compact convex subset. There exist $\tau_0>0$ and $L>0$ such that
$$
\|M_\tau^m[P]-P-\tau \Phi_m(P_0)\|\leq L \tau^ 2+L\tau\|P-P_0\|
$$
for all $P,P_0\in \K$, $\tau\in [0,\tau_0]$ and $m\in \M$.
\end{lemma}
\begin{proof}
Let $\tau_0>0$ such that for all $P\in \K$, $m\in \M$, the Riccati equation 
$$
\dot P= \Phi_m(P),\enspace P(0)=P,
$$
has a solution in $[0,\tau_0]$. Therefore, 
$$\tilde \K:=\{M_t^m[P]:t\in [0,\tau_0], P\in \K,m\in \M\}$$
is compact. Besides, for $P\in \K$ and $m\in \M$, the function 
$
M_{\cdot}^m[P]:[0,\tau_0]\rightarrow \sym_n
$ is twice differentiable in the variable $t$ and it satisfies:
$$
\dot M_t^m[P]=\Phi_m(M_t^m[P]), \quad \ddot M_t^m[P]=D\Phi_m(M_t^m[P])\circ\Phi_m(M_t^m[P]),\quad t\in[0,\tau_0].
$$
By the mean value theorem, for all $P,P_0\in \K$ and $\tau\in[0,\tau_0]$
$$
\|M_\tau^m[P]-P-\tau\Phi_m(P)\|\leq \sup_{t\in(0,\tau)}\|D\Phi_m(M_t^m[P])\circ \Phi_m(M_t^m[P])\| \tau^ 2
$$
and
$$
\|\Phi_m(P)-\Phi_m(P_0)\|\leq \sup_{Q\in \K} \|D\Phi_m(Q)\|\|P-P_0\|.
$$
Let 
$$
L=\max\{\sup_{m}\sup_{P\in \tilde \K} \|D\Phi_m(P)\circ \Phi_m(P)\|,\sup_m \sup_{P\in \K}\|D\Phi_m(P)\|\},
$$
then we have
$$
\|M_\tau^m[P]-P-\tau\Phi_m(P_0)\|\leq L \tau^ 2+L\tau\|P-P_0\|
$$
for all $P,P_0\in \K$, $\tau \in [0,\tau_0]$ and $m\in \M$.
\end{proof}

Using Lemma~\ref{l-TaylorD} we give a proof of Lemma~\ref{l-onestep}:

\begin{proof}[Proof of Lemma~\ref{l-onestep}]

 Let any $0<\delta<1$ and $\tilde \K\subset \sym_n$ be the compact convex defined by:
\begin{displaymath}
\tilde \K:=\displaystyle \overline\conv(\cup_{P_0\in \K} \overline {B(P_0,\delta)}).
\end{displaymath}
By Lemma~\ref{l-TaylorD}, there exists $\tau_1, L_1>0$ such that for all $m\in \M$, $P,P_0\in \tilde\K$ and $\tau \in [0,\tau_1]$
\begin{align}\label{adgfgfg}
M_\tau^m[P]\leq P+\tau \Phi_m(P_0)+(L_1\tau^2+L_1\tau\|P-P_0\|)I.
\end{align}
Let 
$$
\begin{array}{l}
L_2=\sup \{\|\Phi_m(P)\|: m\in \M,P\in \K\},\\ L_0=\max(L_1,L_1L_2), \\\tau_0=\min(\frac{\delta}{2L_2},\sqrt{\frac{\delta}{2eL_0}},\frac{1}{L_1},\tau_1).
\end{array}
$$
Let any $N\in \mathbb{N}$ ,$(i_1,\cdots,i_N)\in \M^N$, $\tau\in[0,\tau_0]$ and $V^0(x)=\frac{1}{2}x'P_0x$ with $P_0\in \K$. We are going to prove by induction on $k\in \{1,\cdots,N\}$ that:
\begin{align}\label{bgsf}
M^{i_k}_{\tau/N}\dots M^{i_1}_{\tau/N}[P_0]\leq P_0+\frac{\tau}{N}\Phi_{i_1}(P_0)+\cdots+\frac{\tau}{N}\Phi_{i_k}(P_0)+L_0(1+\frac{1}{N})^k \frac{\tau^2k^2}{N^2}I
\end{align}
When $k=1$, since $\frac{\tau}{N}\in[0,\tau_0]$ and $P_0\in \tilde\K$, by~\eqref{adgfgfg} we get:
$$
\begin{array}{ll}
M_{\tau/N}^{i_1}(P_0)&\leq P_0+\frac{\tau}{N}\Phi_{i_1}(P_0)+L_1(\frac{\tau}{N})^2 I\\
&\leq P_0+\frac{\tau}{N}\Phi_{i_1}(P_0)+L_0(1+\frac{1}{N})(\frac{\tau}{N})^2 I.
\end{array}
$$
Suppose that~\eqref{bgsf} is true for some $k\in \{1,\cdots,N-1\}$. That is:
\begin{align}\label{adfgeer}
M^{i_k}_{\tau/N}\cdots M^{i_1}_{\tau/N}[P_0]\leq P_0+\Delta_k
\end{align}
where $\Delta_k=\frac{\tau}{N}\Phi_{i_1}(P_0)+\cdots+\frac{\tau}{N}\Phi_{i_k}(P_0)+L_0(1+\frac{1}{N})^k \frac{\tau^2k^2}{N^2}I.$
Since 
$$
\begin{array}{ll}
\|\Delta_k\|&\leq \frac{k\tau}{N}L_2+L_0(1+\frac{1}{N})^k \frac{\tau^2k^2}{N^2}\\
&\leq \tau L_2+L_0e\tau^2\leq \delta,
\end{array}
$$
we have that $P_0+\Delta_k\in \tilde\K$ and by~\eqref{adgfgfg}:
$$
\begin{array}{ll}
M_{\tau/N}^{i_{k+1}}[P_0+\Delta_k]&\leq P_0+\Delta_k+\frac{\tau}{N}\Phi_{i_{k+1}}(P_0)+(L_1\frac{\tau^2}{N^2}+L_1\frac{\tau}{N} \|\Delta_k\|)I\\
&\leq P_0+\frac{\tau}{N}\Phi_{i_1}(P_0)+\cdots+\frac{\tau}{N}\Phi_{i_k}(P_0)+L_0(1+\frac{1}{N})^k \frac{\tau^2k^2}{N^2}I\\
&\quad+\frac{\tau}{N}\Phi_{i_{k+1}}(P_0)
+L_1\frac{\tau^2}{N^2}I+L_1\frac{\tau}{N}[\frac{k\tau}{N}L_2+L_0(1+\frac{1}{N})^k \frac{\tau^2k^2}{N^2}]I\\
&=P_0+\frac{\tau}{N}\Phi_{i_1}(P_0)+\cdots+\frac{\tau}{N}\Phi_{i_k}(P_0)+\frac{\tau}{N}\Phi_{i_{k+1}}(P_0)\\
&\quad \frac{\tau^2}{N^2}[L_0(1+\frac{1}{N})^kk^2+L_1+L_1L_2k+L_1L_0(1+\frac{1}{N})^k\frac{\tau k^2}{N}]I\\
&\leq P_0+\frac{\tau}{N}\Phi_{i_1}(P_0)+\cdots+\frac{\tau}{N}\Phi_{i_k}(P_0)+\frac{\tau}{N}\Phi_{i_{k+1}}(P_0)\\
&\quad +\frac{\tau^2}{N^2}[L_0(1+\frac{1}{N})^k(k^2+k+1)+L_0(1+\frac{1}{N})^k\frac{k^2}{N}]I\\
&\leq P_0+\frac{\tau}{N}\Phi_{i_1}(P_0)+\cdots+\frac{\tau}{N}\Phi_{i_k}(P_0)+\frac{\tau}{N}\Phi_{i_{k+1}}(P_0)\\&\quad+\frac{\tau^2(k+1)^2}{N^2}L_0(1+\frac{1}{N})^{k+1}I
\end{array}
$$
Thus, by~\eqref{adfgeer} and the monotonicity of the flow:
$$
\begin{array}{ll}
&M^{i_{k+1}}_{\tau/N}M^{i_k}_{\tau/N}\cdots M^{i_1}_{\tau/N}[P_0]\leq M_{\tau/N}^{i_{k+1}}[P_0+\Delta_k]\\
&\leq P_0+\frac{\tau}{N}\Phi_{i_1}(P_0)+\cdots+\frac{\tau}{N}\Phi_{i_k}(P_0)+\frac{\tau}{N}\Phi_{i_{k+1}}(P_0)+L_0(1+\frac{1}{N})^{k+1}\frac{\tau^2(k+1)^2}{N^2}I.
\end{array}
$$
We conclude that:
$$
\begin{array}{ll}
M^{i_{N}}_{\tau/N}\cdots M^{i_1}_{\tau/N}[P_0]&\leq P_0+\frac{\tau}{N}\Phi_{i_1}(P_0)+\cdots+\frac{\tau}{N}\Phi_{i_N}(P_0)+eL_0\tau^2 I
\end{array}
$$
Denote:
 $$
g(x)=\sup_{m\in \M}\frac{1}{2}( x'P_0x+x'\Phi_m(P_0)x).
$$
By Lemma~\ref{l-TaylorD} we have that
$$
P_0+\tau \Phi_m(P_0)\leq M_\tau^m[P_0]+L_1\tau^2 I, \quad \forall \tau \in[0,\tau_0], m\in \M.
$$
That is
$$
g(x)\leq S_\tau^m[V^0](x)+\frac{L_1}{2}\tau^2 |x|^2 , \quad \forall \tau \in[0,\tau_0], m\in \M.
$$
Therefore,
$$
\begin{array}{ll}
S^{i_{N}}_{\tau/N}\cdots S^{i_1}_{\tau/N}[V^0](x)
&=\frac{1}{2} x'M^{i_{N}}_{\tau/N}\cdots M^{i_1}_{\tau/N}[P_0]x \\ &\leq \frac{1}{2} (x'P_0x+\frac{\tau}{N}x'\Phi_{i_1}(P_0)x+\cdots+\frac{\tau}{N}x'\Phi_{i_N}(P_0)x+eL_0\tau^2 |x|^2)
\\
&\leq g(x)+ \frac{e L_0}{2}\tau^2|x|^2 \\
 &\leq \sup_m S_\tau^m[V^0](x)+L\tau^2 |x|^2,\enspace \forall x\in \R^n
\end{array}
$$
where $L=\frac{eL_0+L_1}{2}$ is clearly independent of $V^0$, $N$, $(i_1,\cdots,i_N)$ and $\tau\leq \tau_0$. We conclude that:
$$
\sup_N\sup_{i_1,\cdots,i_N}S^{i_{N}}_{\tau/N}\cdots S^{i_1}_{\tau/N}[V^0](x)\leq \sup_{m} S_\tau^m[V^0](x)+L\tau^2 |x|^2
$$
for all $\tau\in[0,\tau_0]$ and $V^0(x)=x'P_0x$ with $P_0\in \K$.
Finally we apply Lemma~\ref{lemma-St} to obtain the desired result.
\end{proof}

\section{Further discussions and a numerical illustration}\label{sec-dis}

\subsection{Linear quadratic Hamiltonians}
The contraction result being crucial to our analysis~(see Remark~\ref{rem-crucial}), it is impossible to extend the results to the general case with linear terms as in~\cite{MR2558316}. However, the one step error analysis~(Lemma~\ref{l-onestep}) is not restricted to the pure quadratic Hamiltonian. Interested reader can verify that the one step error $O(\tau^2)$ still holds in the case of~\cite{MR2558316}. Then by simply adding up the errors to time $T$, we get that:
$$\epsilon(x,\tau,N,V^0)\leq L(1+|x|^2)N\tau^2=L(1+|x|^ 2)T\tau.$$
Note that the term $|x|^2$ is replaced by $(1+|x|^2)$ for the general Hamiltonian with linear terms.
This estimate is of the same order as in~\cite{MR2558316} with much weaker assumption, especially the assumption on $\Sigma^m$.  

\subsection{A tighter bound on the complexity}

From Theorem~\ref{theo-error1} and~\ref{theo-error2}, we obtain a tighter bound on the complexity of the algorithm (compared to~\cite{MR2599910}):
\begin{coro}\label{coro-numberofiter}
Under Assumptions~\ref{assum1} and~\ref{assum2}, to get an approximation of $V$ of order $\epsilon$, the number of iterations is 
$$
O(\frac{-\log \epsilon}{\epsilon}),\enspace \mathrm{~as~~} \epsilon\rightarrow 0,
$$
whence the number of arithmetic operations is:
\begin{align}\label{a-boundop}
O(|\M|^{O(\frac{-\log \epsilon}{\epsilon})}n^3),\enspace \mathrm{~as~~} \epsilon\rightarrow 0.
\end{align}
\end{coro}

\subsection{Convergence time}
Theorem~\ref{theo-error2} shows that for a sufficiently small $\tau$ and a pruning procedure of error $\tau^2$, the discrete-time approximation error $\epsilon^{\pP_\tau}(x,\tau, N, V^0)$ is $O(\tau)$ uniformly for all $N>0$. 
Meanwhile, by Theorem~\ref{theo-error1}, the finite horizon approximation error $\epsilon_0(x,T,V^0)$ decreases exponentially with the time horizon $T$.
Therefore, for a fixed sufficiently small $\tau$, the total error $\epsilon_0(x,T,V^0)+\epsilon^{\pP_\tau}(x,\tau, N, V^0)$ decreases at each propagation step and becomes stationary after a time horizon $T$ such that \begin{align}\label{a-epsilonT}\epsilon_0(x,T,V^0)\leq \epsilon^{\pP_\tau}(x,\tau, N, V^0).
\end{align}
 If the estimate $O(\tau)$ is tight, then~\eqref{a-epsilonT} implies that the stationary time $T$ is bounded by the relation:
$$
O(e^{-\alpha T})\leq O(\tau).
$$
Therefore the stationary time $T$ is bounded by $O(-\log(\tau))$,
which implies that numerically the total error stops decreasing after a number of iterations $N$ bounded by $O(-\log(\tau)/\tau)$.

To give an illustration, we implemented this max-plus approximation method, incorporating a pruning algorithm in~\cite{conf/cdc/GaubertMQ11} to a problem instance satisfying Assumption~\ref{assum1} and~\ref{assum2}
in dimension $n=2$ and with $|\M|=3$ switches. The pruning algorithm generates an error of order at most $\tau^2$ at each step.
We use the maximal absolute value of $H(x,\nabla V)$ \eqref{a-define-H} on the region $[-2,2]\times[-2,2]$ as the back-substitution error, denoted by $|H|_\infty$, to measure the approximation. We observe that for each $\tau$, the back-substitution error $|H|_\infty$ becomes stationary after a number of iterations, see Figure~\ref{fl1} for $\tau=0.0006$.
\begin{figure}[thpb]
      \centering
      \includegraphics[scale=0.23]{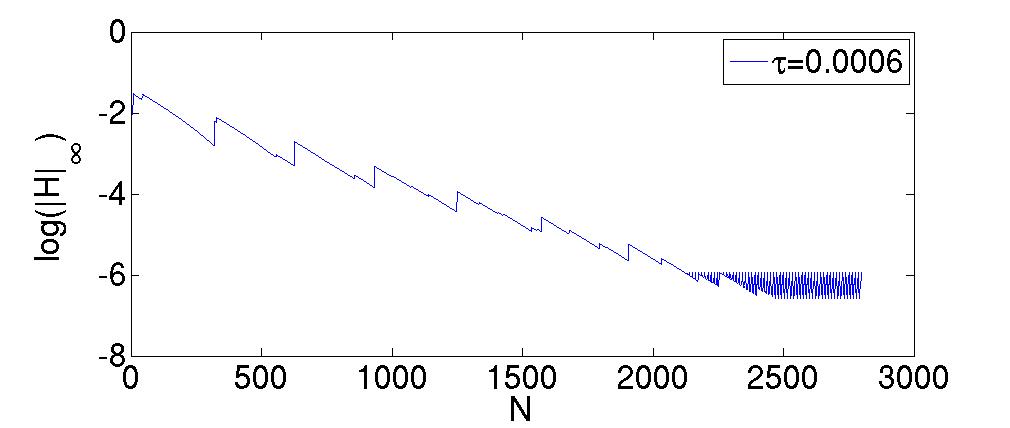}
      \caption{Plot of $\log |H|_\infty$ w.r.t. the iteration number $N$}
      \label{fl1}
   \end{figure}
We run the instance for different $\tau$ and for each $\tau$ we collect the time horizon $T$ when the 
back-substitution error becomes stationary. The plot shows a linear growth of $T$ with respect to $-\log(\tau)$,
which is an illustration of the exponential decreasing rate in Theorem~\ref{theo-error1}.

\begin{figure}[thpb]
      \centering
      \includegraphics[scale=0.23]{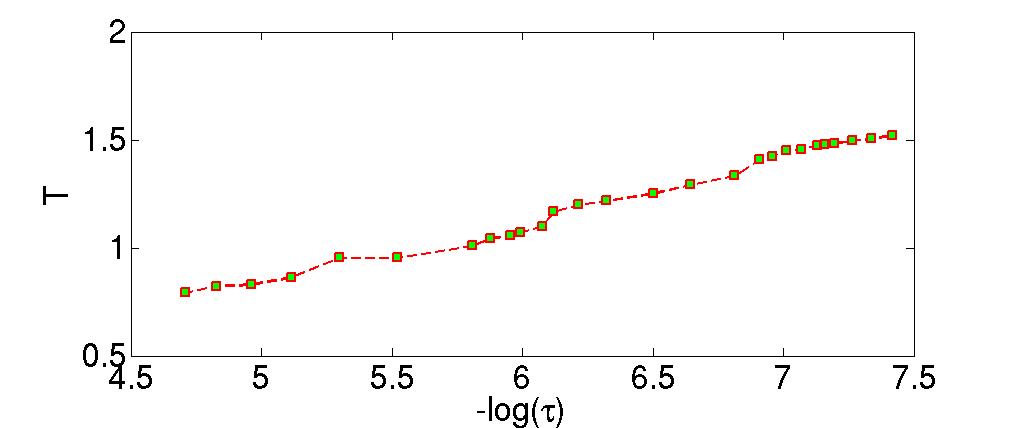}
      \caption{Plot of the convergence time $T$ w.r.t $-\log(\tau)$}
      \label{fl2}
   \end{figure}

\section*{Acknowledgments}
The author thanks Prof. S. Gaubert for his important suggestions and guidance on the present work.

\bibliographystyle{alpha}
\bibliography{biblio}

\end{document}